\documentclass[11pt]{article}

\usepackage[T1]{fontenc}
\usepackage{lmodern}

\usepackage{amsmath, amssymb}
\usepackage[truedimen,nohead,includefoot,
hmargin=2cm,top=2cm,bottom=1cm,footskip=1cm]{geometry}
\usepackage{bm}
\usepackage{mathtools}
\usepackage{framed}
\usepackage{color}
\usepackage{graphicx}
\usepackage[dvipsnames]{xcolor}
\usepackage{tabls}
\usepackage{multirow}
\usepackage{subcaption}
\usepackage{tikz}
\usetikzlibrary{calc}
\usepackage{enumitem}
\usepackage{amsthm}
\usepackage{caption}

\theoremstyle{plain}
\newtheorem{thm}{Theorem}[section]

\newtheorem{lem}[thm]{Lemma}

\theoremstyle{remark}
\newtheorem{rem}[thm]{Remark}

\theoremstyle{definition}
\newtheorem{defi}[thm]{Definition}

\mathtoolsset{showonlyrefs=true}
\makeatletter

\@addtoreset{equation}{section}
\makeatother
\allowdisplaybreaks

\DeclareMathOperator*{\minimize}{minimize}
\newcommand{\bR}{\mathbb{R}}
\newcommand{\bN}{\mathbb{N}}

\title{
Numerical analysis of elastica with obstacle and adhesion effects
}
\author{Tomoya \textsc{Kemmochi}%
\thanks{Graduate School of Mathematical Sciences, The University of Tokyo,
Komaba 3-8-1, Meguro-ku, Tokyo, 153-8914, Japan.
\textit{E-mail}: \texttt{kemmochi@ms.u-tokyo.ac.jp}}
}

\begin{document}
\maketitle

\begin{abstract}
We consider the numerical computation of a variational problem that arises from materials science.
The target functional is a type of elastic energy that is influenced by obstacles and adhesion.
Owing to its strong nonlinearity and discontinuity, the Euler-Lagrange equation is very complicated, and numerical computation of its critical points is difficult.
In this paper, we discretize and regularize the target energy as a functional defined on a space of polygonal curves.
Moreover, we develop convergence analysis for discrete minimizers in the framework of $\Gamma$-convergence.
We first show that the discrete energy functional $\Gamma$-converges to the original one.
Then, we establish the compactness property for the sequence of discrete minimizers.
These two results allow us to extract a convergent subsequence from the discrete minimizers.
We also present some numerical examples in the last section of the paper.
Existence of singular local minimizers is suggested by numerical experiments.

\bigskip

\noindent\textbf{Keywords:} elastic energy; elastica; finite difference method, $\Gamma$-convergence, obstacle problem, adhesive problem

\medskip

\noindent\textbf{Mathematics Subject Classification (MSC2010):}
74G15; 
74G65; 

\end{abstract}

\section{Introduction}
\label{section:intro}
In this paper, we consider the following variational problem arising from materials science:
\begin{equation}
\minimize_{u \ge \psi} E[u] :=
\frac{C}{2} \int \kappa^2 ds + \sigma \int ds
- \gamma \int_{\{u = \psi \}} ds,
\label{eq:problem0}
\end{equation}
where
$u \colon (0,1) \to \bR$ is an unknown function, $\psi \colon (0,1) \to \bR$ is a given smooth function,
$C$, $\sigma$, and $\gamma$ are given positive constants, and $\kappa$ and $ds$ are, respectively, the curvature and the line element of the graph of $u$.
We impose the periodic boundary condition at $x=0$ and $x=1$ (that is, we assume that $u$ and $\psi$ are periodic functions on $(0,1)$).
In the problem \eqref{eq:problem0}, the graph of $u$ expresses the shape of a membrane or a filament, and that of $\psi$ expresses a rippled surface.
The first term is called bending energy, which straightens the filament.
The second is called tension, which shortens the filament.
The third is called adhesion energy, which forces the filament to adhere to the obstacle.
The functional $E$ is the surface energy of the membrane, and its minimizer describes the steady state of the membrane constrained above the obstacle.
We consider only the one-dimensional case, that is, when $u$ is a filament or a membrane depending on only one direction.
For more details on the physical background of the problem \eqref{eq:problem0}, refer to \cite{Pie08}, who proposed this problem.

The energy functional
\begin{equation}
\frac{C}{2} \int \kappa^2 ds + \sigma \int ds
\label{eq:elastic-energy}
\end{equation}
is called elastic energy.
Its critical point is known as elastica, first introduced by Euler \cite{Eul1744} and then studied analytically (see \cite{Lov44,Sin08}) and numerically (see \cite{BruHN01,BruNR01,IglB15}) by many researchers.
In addition to the elastic energy, we consider the effects of obstacles and adhesion in the problem \eqref{eq:problem0}.
Although there are several studies on the obstacle problem with the effect of adhesion (see for example \cite{AltC81,Oma93,Yam03}),
they considered the Dirichlet energy, rather than our elastic energy.

In \cite{Miu16}, the problem \eqref{eq:problem0} is studied analytically.
To our knowledge, there is no mathematical result on the problem \eqref{eq:problem0} other than \cite{Miu16}.
The singular perturbation problem as $C \downarrow 0$ is considered in \cite{Miu16}.
However, even the solvability of the problem \eqref{eq:problem0} has as yet not been established.
In view of the materials scientific background mentioned above, it is important to compute minimizers of the functional $E$ numerically.
Therefore, for the convergence analysis of the computed minimizers, it is desirable to guarantee the existence of solutions of the problem \eqref{eq:problem0}.

Our study has three aims:
\begin{enumerate}[label=(\Alph*), leftmargin=*]
\item derivation of the solvability of the problem \eqref{eq:problem0},
\item numerical computation of the minimizers of $E$, and
\item convergence analysis of discrete minimizers.
\end{enumerate}
Setting aside the question of solvability, one might think that it is not difficult to compute an approximate solution of \eqref{eq:problem0}.
However, it is very difficult to derive the Euler-Lagrange equation of the functional $E$, owing to its strong nonlinearity and discontinuity (resulting from the adhesion term).
Moreover, since the functional $E$ is not convex, the solution might not be unique even if it exists.
Therefore, we will forget the Euler-Lagrange equation and rely more on direct numerical computation and convergence analysis instead.

Our strategy is as follows.
We first discretize and regularize the functional $E$ as a continuous functional $E_{h,\delta,\rho}$ (see \eqref{eq:disc-energy}) on the space of periodic polygonal curves defined on the interval $(0,1)$ (for precise notation, see Subsection~\ref{subsec:discretization}).
Moreover, we add a penalty term to handle the effects of obstacles.
Then, the discretized functional is a continuous function defined on $\mathbb{R}^N$, and thus, can be handled by many existing numerical optimization methods, for example, the quasi-Newton method.
Hence, we can accomplish task (B).
For convergence analysis, we use the method of $\Gamma$-convergence,
which is convergence of functionals.
The notion of $\Gamma$-convergence was first introduced by De Giorgi \cite{DeGF75} in the 1970s, and numerous results on variational problems have been established in this framework since then.
The definition of $\Gamma$-convergence is as follows.
\begin{defi}[$\Gamma$-convergence]\label{def:gamma-conv}
Let $X$ be a metric space, and let $F, F_\varepsilon \colon X \to \bR \cup \{+\infty\}$ be functionals defined on $X$.
We say that $F_\varepsilon$ \textit{$\Gamma$-converges} to $F$ as $\varepsilon \to 0$ if the following two properties hold:
\begin{description}
\item[(U)]
$\forall x \in X$, $\exists x_\varepsilon \to x$  in $X$ s.t.
$\limsup_{\varepsilon \to 0} F_\varepsilon[x_\varepsilon] \le F[x]$.
\item[(L)]
$\forall x \in X$, $\forall x_\varepsilon \to x$  in $X$,
$\liminf_{\varepsilon \to 0} F_\varepsilon[x_\varepsilon] \ge F[x]$.
\end{description}
\end{defi}
In this paper, the following fundamental theorem of $\Gamma$-convergence plays a crucial role.
\begin{lem}\label{lem:fundamental}
Let $X$ be a metric space and let $F_\varepsilon$ be a functional on $X$ that $\Gamma$-converges to a functional $F$ as $\varepsilon \to 0$.
Assume that each $F_\varepsilon$ admits at least one minimizer $\bar{x}_\varepsilon$, and that the sequence $\{ x_\varepsilon \}_\varepsilon$ has a cluster point $\bar{x}$.
Then, $F$ attains a minimum at $\bar{x}$.
Moreover,
\begin{equation}
F[\bar{x}] = \inf_{x \in X}F[x] = \lim_{\varepsilon \to 0} F_\varepsilon[\bar{x}_\varepsilon].
\end{equation}
\end{lem}
For more details on $\Gamma$-convergence, see \cite{DalM93,Bra02,Bra14}.
The notion of $\Gamma$-convergence and its fundamental theorem guarantee the existence of minimizers of the functional $F$ when we do not know whether $F$ attains a minimum.
Moreover, they can be applied to the numerical analysis of variational problems.
Indeed, if we intend to compute a minimizer of a given functional $F$, we can take the following steps:
\begin{enumerate}[label=Step \arabic*., leftmargin=*]
\item We discretize the functional $F$ as functionals $F_h$ defined on finite-dimensional spaces.
\item We show that the sequence $\{ F_h \}_h$ $\Gamma$-converges to the functional $F$.
\item We show that each functional $F_h$ has at least one minimizer $\bar{x}_h$.
\item We show that the sequence $\{ \bar{x}_h \}_h$ has a cluster point $\bar{x}$.
\end{enumerate}
At this stage, we can establish that $F$ attains a minimum at $\bar{x}$, and we can extract a subsequence from $\{ \bar{x}_h \}_h$ that converges to the original solution $\bar{x}$, as a result of Lemma \ref{lem:fundamental}.
It is remarkable that this technique is applicable to a problem whose solution is not unique, such as \eqref{eq:problem0}.
There are several studies that develop the above method (cf.~\cite{BelC94,BalB91,AubBM04,CasZ11,BarMR12}).
We also follow the above strategy to accomplish goals (A) and (C).

Our principal results are the $\Gamma$-convergence and compactness results.
We show that the functional $E_{h,\delta,\rho}$ $\Gamma$-converges to $E$ in the topology of $H^1$, under the condition that $W^{1,\infty}$-norms are bounded (Theorem~\ref{thm:gamma}).
We also show that the sequence of minimizers of $\{ E_{h,\delta,\rho} \}_{h,\delta,\rho}$ has a cluster point under the same constraint as in the $\Gamma$-convergence result (Theorem~\ref{thm:compactness-S}).
It is essential in our proof that $W^{1,\infty}$-norms are bounded, especially in the proof of the compactness result.
Therefore, our principal results are ``local optimization'' in a sense.
We provide a sufficient condition for the global optimization of $E$ (Theorem \ref{thm:sufficient}).
However, we propose a conjecture that the functional $E$ might have no global minimizer in general (see numerical examples in Subsection~\ref{subsec:singular} and Remark \ref{rem:singular}).

The rest of the paper is organized as follows.
In Section~\ref{section:preliminary}, we present some preliminary results.
First, we formulate the problem \eqref{eq:problem0} in the framework of $\Gamma$-convergence, and then, we discretize the functional $E$.
We show some basic lemmas on the finite difference operators in the last part of Section~\ref{section:preliminary}.
The principal theorems are discussed in Sections~\ref{section:gamma-conv} and \ref{section:compactness}.
According to Definition~\ref{def:gamma-conv}, we need two properties (U) and (L) for $\Gamma$-convergence.
We prove property (U) in the early part of Section~\ref{section:gamma-conv} and property (L) in the last part.
The compactness result and a sufficient condition for global optimization are established in Section~\ref{section:compactness}.
Some numerical examples are presented in Section~\ref{section:num_exp}, and we conclude with some remarks in Section~\ref{section:conclusion}.
\section{Preliminaries}
\label{section:preliminary}
In this section, we formulate the original problem in terms of variational problems.
Then, we discretize the problem using the finite difference method.
We provide some basic properties of the finite difference operator in the last part of this section.

\subsection{Model problem and formulation}
\label{subsec:formulation}

We formulate the problem \eqref{eq:problem0} in terms of a variational problem.
We first set the function spaces as follows:
\begin{align}
H^1_\pi &:= \{ u \in H^1(0,1) \mid u(0) = u(1) \}, \\
H^2_\pi &:= \{ u \in H^2(0,1) \cap u \in H^1_\pi \mid u' \in H^1_\pi \} , \\
X^m &:= \{ u \in H^m_\pi \mid u \ge \psi \},
\quad m=1,2.
\end{align}
Then, we redefine a functional $E \colon H^1_\pi \to \bR \cup \{ +\infty \}$ as
\begin{equation}
E[u] := \begin{cases}
B[u] + T[u] - A[u], & u \in X^2, \\
+ \infty , & u \in H^1_\pi \setminus X^2,
\end{cases}
\end{equation}
where
\begin{equation}
B[u] := \frac{C}{2} \int \kappa^2 ds, \quad
T[u] := \sigma \int ds, \quad
A[u] := \gamma \int_{\{ u=\psi \}} ds.
\end{equation}
Note that $E[u]$ is finite for $u \in X^2$.
Now, we can formulate the problem \eqref{eq:problem0} as a minimization problem on $H^1_\pi$:
\begin{equation}
\minimize_{u \in H^1_\pi} E[u].
\label{eq:cont-problem}
\end{equation}
Unfortunately, we have not established global optimization \eqref{eq:cont-problem}, and we believe that the problem \eqref{eq:cont-problem} might not have a solution without any condition on the physical parameters $\psi$, $C$, $\sigma$, and $\gamma$.
See Remark~\ref{rem:singular} and Section~\ref{section:conclusion}.
According to the ``small slope approximation'' in \cite{Pie08}, we impose the following ``bounded slope condition.''
\begin{equation}
\minimize_{u \in X_S} E[u],
\label{eq:cont-problem-S}
\end{equation}
where
\begin{equation}
X_S := \{ v \in H^1_\pi \cap W^{1,\infty}(0,1) \mid \| v \|_{W^{1,\infty}(0,1)} \le S \},
\end{equation}
for $S>0$. The topology of $X_S$ is the one induced from $H^1_\pi$.
\subsection{Discretization of the problem}
\label{subsec:discretization}
In order to compute numerically a (local or global) minimizer of $E$, we will discretize the problem \eqref{eq:cont-problem}.
Typically, the Euler-Lagrange equation is used for minimization problems, and is then solved by appropriate numerical methods, such as the finite element method.
However, the Euler-Lagrange equation of the functional $E$ is very complicated.
Thus, we will decline the use of this equation, and instead,
compute minimizers directly by discretizing the functional $E$ on a space of polygonal curves.

Let $N \in \bN$, $h=1/N$, $x_j = jh$ ($j=0,1,\dots,N$), and $I_j = (x_{j-1}, x_j) \subset \bR$ ($j=1,2,\dots,N$). We define $V_h$ as the space of periodic polygonal curves with respect to the partition $\{ x_j \}$, that is,
\begin{equation}
V_h = \{ v_h \in C^0_\pi \mid v_h|_{I_j} \text{ is affine } \forall j.  \},
\end{equation}
where $C^0_\pi := \{ v \in C^0[0,1] \mid v(0) = v(1) \}$. We introduce some notation. For $v_h \in V_h$ and $j = 0,1,\dots, N$, let
\begin{equation}
\left.
\begin{array}{c}
v_j = v_h(x_j),
\quad d_j = d_j(v_h) = \dfrac{v_j - v_{j-1}}{h},
\quad D_j = D_j(v_h) = \dfrac{v_{j+1} - 2v_j + v_{j-1}}{h^2}, \\[1ex]
l_j = l_j(v_h)
= \dfrac{h}{2} \sqrt{1 + d_j^2},
\quad \theta_j = \theta_j(v_h) = \arccos
\left[
\dfrac{1+d_j d_{j+1}}{\sqrt{1 + d_j^2} \sqrt{1 + d_{j+1}^2}}
\right].
\end{array}
\right\}
\label{eq:discretize}
\end{equation}
Here, and hereafter, we extend indices periodically, for example, $v_{-1} = v_{N-1}$, $v_{N+1} = v_1$, and so on.
Note that $2l_j$ is the length of $v_h$ on $I_j$ and that $\theta_j$ is the angle between two vectors $(1,d_j)$ and $(1, d_{j+1})$ (cf.\ Figure \ref{fig:polygon}).
\begin{figure}[htb]
\centering
\includegraphics{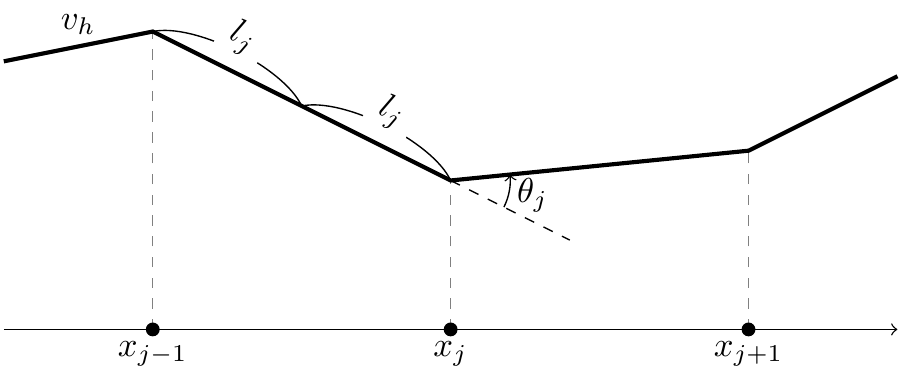}
\caption{Definition of $l_j$ and $\theta_j$.}
\label{fig:polygon}
\end{figure}
Moreover, to avoid the discontinuity caused by the adhesion energy, we regularize the characteristic function.
Fix $\zeta \in C^1(\bR)$ satisfying the following properties (see Figure \ref{fig:regularization}):
\begin{equation}
0 \le \zeta \le 1, \quad \zeta(-t) = \zeta(t), \quad
\zeta|_{(1,\infty)} \equiv 0, \quad
\zeta'|_{[0,1]} \le 0, \quad
\zeta(0) = 1.
\end{equation}
Then, we introduce the regularization parameter $\delta>0$ and set $\zeta_\delta(t) = \zeta(t/\delta)$ ($t \in \bR$).
For $v_h \in V_h$, $|v_j - \psi_j|$ is the distance between $v_h$ and the obstacle $\psi$ at $x_j$, where $\psi_j = \psi(x_j)$.
Let
\begin{equation}
\zeta_{\delta,j} := \zeta_\delta(v_j - \psi_j).
\end{equation}
This describes whether $v_h$ adheres to $\psi$ at $x_j$.
That is, if $|v_j - \psi_j| < \delta (\iff \zeta_{\delta,j} > 0)$, then we judge that $v_h$ ``adheres'' to $\psi$ at $x_j$.
\begin{figure}[htb]
\centering
\includegraphics{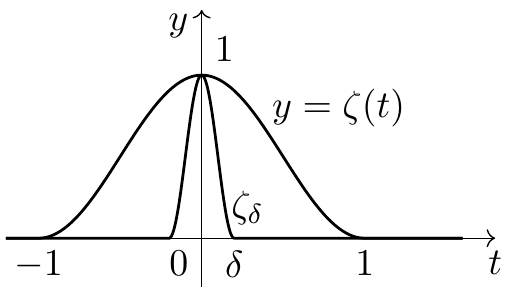}
\caption{Regularization functions $\zeta$ and $\zeta_\delta$.}
\label{fig:regularization}
\end{figure}

Now, we are ready to discretize the functional $E$.
Define the discrete bending energy and discrete adhesion energy as
\begin{equation}
B_h[v_h] := \frac{C}{2} \sum_{j=1}^N \theta_j^2 \frac{l_j^3 + l_{j+1}^3}{l_j l_{j+1} (l_j + l_{j+1})^2},
\quad
A_{h,\delta}[v_h] := \gamma \sum_{j=1}^{N} \zeta_{\delta,j-1} \zeta_{\delta,j} \cdot 2l_j.
\end{equation}
The definition of $B_h$ was introduced by \cite{BruNR01} and \cite{BruHN01}.
Note that
\begin{equation}
\theta_j^2 \frac{l_j^3 + l_{j+1}^3}{l_j l_{j+1} (l_j + l_{j+1})^2}
= \frac{\theta_j^2}{l_j + l_{j+1}}\left( \frac{l_{j+1}}{l_j} -1 + \frac{l_j}{l_{j+1}} \right)
\ge \frac{\theta_j^2}{l_j + l_{j+1}}.
\label{eq:bending-0}
\end{equation}
This relation is used several times later.
In addition, we introduce the penalty term
\begin{equation}
P_{h,\rho}[v_h] := \frac{1}{\rho} \sum_{j=1}^N |(v_j - \psi_j)_-|^2 h,
\end{equation}
where the penalty parameter $\rho>0$ is a small positive number and $(x)_- = \min\{ 0, -x \}$ is the negative part of the real number.
Then, we define the discrete energy functional $E_{h,\delta,\rho}$ as
\begin{equation}
E_{h,\delta,\rho}[v] :=
\begin{cases}
B_h[v] + T[v] - A_{h,\delta}[v] + P_{h,\rho}[v], & v \in V_h, \\
+\infty, & v \in H^1_\pi \setminus V_h
\end{cases}
\label{eq:disc-energy}
\end{equation}
for $v \in H^1_\pi$.
The discrete problems corresponding to \eqref{eq:cont-problem} and \eqref{eq:cont-problem-S} are, respectively, formulated as
\begin{equation}
\minimize_{u \in H^1_\pi} E_{h,\delta,\rho}[u]
\label{eq:disc-problem}
\end{equation}
and
\begin{equation}
\minimize_{u \in X_S} E_{h,\delta,\rho}[u]
\label{eq:disc-problem-S}
\end{equation}
for $S>0$.
Since $V_h \cap X_S$ is compact and $E_{h,\delta,\rho}$ is continuous on $V_h$, the problem \eqref{eq:disc-problem-S} has at least one solution.
\begin{lem}\label{lem:boundedness}
For each $S>0$, the functional $E_{h,\delta,\rho}$ admits at least one minimizer in $X_S$.
Moreover, if there exists $c_0 > 0$ such that $\delta \le c_0 h$, then the minimizer $\bar{v}^{(S)}_{h,\delta,\rho} \in V_h$ satisfies
\begin{equation}
\| \bar{v}^{(S)}_{h,\delta,\rho} \|_{L^\infty(0,1)} \le M
\end{equation}
for some $M>0$, which is independent of $h$, $\delta$, $\rho$, and $S$.
\end{lem}

\begin{proof}
The existence of a minimizer has already been observed.
Let us prove the boundedness of the minimizer by contradiction.
We can assume $\bar{v}^{(S)}_{h,\delta,\rho} \ge \psi \ge 0$ by adding some constants if necessary.
We first suppose that
\begin{equation}
\min_{x \in (0,1)} \bar{v}^{(S)}_{h,\delta,\rho}(x)
\ge \max_{x \in (0,1)} \psi(x) + \delta.
\end{equation}
Then, we can find a constant $c>0$ that satisfies
\begin{equation}
E_{h,\delta,\rho}[\bar{v}^{(S)}_{h,\delta,\rho} - c]
< E_{h,\delta,\rho}[\bar{v}^{(S)}_{h,\delta,\rho}],
\end{equation}
which is impossible.
Thus,
\begin{equation}
\min_{x \in (0,1)} \bar{v}^{(S)}_{h,\delta,\rho}(x)
< \max_{x \in (0,1)} \psi(x) + \delta.
\label{eq:minimum-1}
\end{equation}

Let $J_{h,\delta} := \{ j=1,\dots,N \mid \zeta_{\delta,j-1} \zeta_{\delta,j} > 0 \}$.
Then, for every $v_h \in V_h$ and $j \in J_{h,\delta}$, we have
\begin{equation}
|v_h(x_j) - v_h(x_{j-1})| \le |\psi(x_j) - \psi(x_{j-1})| + 2\delta.
\label{eq:sufficient-1}
\end{equation}
Therefore, the condition $\delta \le c_0 h$ implies
\begin{align}
-A_{h,\delta}[v_h]
&\le \gamma \sum_{j \in J_{h,\delta}} \sqrt{(|\psi(x_j)- \psi(x_{j-1})| + 2\delta)^2 + h^2} \\
&\le 4\gamma \sum_{j=1}^N \left[ \sqrt{(\psi(x_j) - \psi(x_{j-1}))^2 + h^2} + \delta \right] \\
&\le 4\gamma (T[\psi]/\sigma + c_0)
\label{eq:sufficient-2}
\end{align}
for $v_h \in V_h$.
Now, let us assume that for each $M>0$, there exist $h$, $\delta$, $\rho$, and $S$ such that
\begin{equation}
\| \bar{v}^{(S)}_{h,\delta,\rho} \|_{L^\infty(0,1)} =
\max_{x \in (0,1)} \bar{v}^{(S)}_{h,\delta,\rho}(x) > M.
\label{eq:minimum-2}
\end{equation}
Then, noting that
\begin{equation}
T[\bar{v}^{(S)}_{h,\delta,\rho}] \ge \sigma \left(
\max_{x \in (0,1)} \bar{v}^{(S)}_{h,\delta,\rho}(x) -
\min_{x \in (0,1)} \bar{v}^{(S)}_{h,\delta,\rho}(x) \right),
\end{equation}
we have
\begin{equation}
E_{h,\delta,\rho}[\bar{v}^{(S)}_{h,\delta,\rho}]
\ge T[\bar{v}^{(S)}_{h,\delta,\rho}] - A_{h,\delta}[\bar{v}^{(S)}_{h,\delta,\rho}]
\ge M - \| \psi \|_{L^\infty(0,1)} - \delta - 4\gamma (T[\psi]/\sigma + c_0)
\end{equation}
by \eqref{eq:minimum-1}, \eqref{eq:sufficient-2}, and \eqref{eq:minimum-2}.
Thus, if we choose
\begin{equation}
M = 1+\sigma + \| \psi \|_{L^\infty(0,1)} + \delta + 4\gamma (T[\psi]/\sigma + c_0),
\end{equation}
then
\begin{equation}
E_{h,\delta,\rho}[\bar{v}^{(S)}_{h,\delta,\rho}]
\ge 1 + \sigma = 1 + E_{h,\delta,\rho}[c_h]
\end{equation}
for some constant function $c_h \in V_h$, which is a contradiction.
Hence, we can complete the proof.
\end{proof}

\subsection{Preliminaries on finite difference operators}
\label{subsec:FD-operators}
We introduce some basic properties of finite difference operators in one dimension.
We first define the operators $d_{h,j}$ and $D_{h,j}$ as
\begin{align}
d_{h,j}v &:= \frac{v(x_j) - v(x_{j-1})}{h}, \\
D_{h,j}v &:= \frac{v(x_{j+1}) -2v(x_j) + v(x_{j-1})}{h^2},
\end{align}
for $j=1,\dots,N$, $v \in C^0_\pi$.
\begin{lem}\label{lem:fdm}
The following statements hold.
\begin{enumerate}
\item If $v \in H^1_\pi$, then for each $j$,
\begin{equation}
|d_{h,j} v| \le h^{-1/2} \| v' \|_{L^2(I_j)}.
\label{eq:lem-fdm-1}
\end{equation}
\item If $v \in H^2_\pi$, then for each $j$ and $x \in I_j$,
\begin{equation}
|d_{h,j}v - v'(x)| \le \sqrt{\frac{2}{3}} h^{1/2} \| v'' \|_{L^2(I_j)}.
\label{eq:lem-fdm-2}
\end{equation}
\item If $v \in H^2_\pi$, then
\begin{equation}
\sum_{j=1}^{N} |D_{h,j} v|^2 h
\le \frac{4}{3} \| v'' \|_{L^2(0,1)}^2.
\label{eq:lem-fdm-3}
\end{equation}
\end{enumerate}
\end{lem}

\begin{proof}
The first assertion \eqref{eq:lem-fdm-1} is a simple consequence of the fundamental theorem of calculus and the H\"older inequality.
We show the assertions \eqref{eq:lem-fdm-2} and \eqref{eq:lem-fdm-3}.
From the fundamental theorem of calculus and Fubini's theorem, we have
\begin{align}
d_{h,j}v - v'(x) &= \frac{1}{h} \int_{x_{j-1}}^{x_j} \phi(y) v''(y) dy, &
\phi(y) &= \begin{cases}
x_{j-1} - y, & x_{j-1} < y < x, \\
x_j - y, & x < y < x_j,
\end{cases} \\
D_{h,j}v &= \frac{1}{h^2} \int_{x_{j-1}}^{x_{j+1}} \psi(y) v''(y) dy,&
\psi(y) & = \begin{cases}
y - x_{j-1}, & y \in I_j, \\
x_{j+1} - y, & y \in I_{j+1}.
\end{cases}
\label{eq:D_j}
\end{align}
Hence, the H\"older inequality implies the desired estimate.
\end{proof}

\section{$\Gamma$-convergence of the discrete functional}
\label{section:gamma-conv}

In this section, we show that the discrete functional $E_{h,\delta,\rho}$ $\Gamma$-converges to the original functional $E$, which is introduced in Definition~\ref{def:gamma-conv}.
The main result of this section is the following $\Gamma$-convergence theorem.
\begin{thm}\label{thm:gamma}
Let $S>0$.
The functional $E_{h,\delta,\rho}$ $\Gamma$-converges to $E$ as $h,\delta,\rho \downarrow 0$ in the topology of $X_S$.
\end{thm}

\subsection{Lower-order terms}
Let us first consider the lower-order terms.
In the following discussion, we denote the Lagrange interpolation with respect to the nodes $\{ x_j \}$ by $\Pi_h$.
\begin{lem}\label{lem:tension}
The functional $T$ is continuous in $H^1_\pi$.
\end{lem}

\begin{proof}
The assertion follows from the identity
\begin{equation}
|\sqrt{1+s^2} - \sqrt{1+t^2}| \le |s-t|
\label{eq:sqrt}
\end{equation}
and the H\"older inequality.
\end{proof}

\begin{lem}\label{lem:adhesion}
The functional $A_{h,\delta}$ satisfies the following two assertions.
\begin{enumerate}
\item Let $v \in X^2$ and $v_{h,\delta} \in V_h$ with $v_{h,\delta} \to u$ in $H^1_\pi$.
Then,
\begin{equation}
\liminf_{h \downarrow 0} (-A_{h,\delta}[v_{h,\delta}]) \ge -A[v].
\label{eq:lsc-A}
\end{equation}
\item If $v \in X^2$, then $\lim_{h \downarrow 0} A_{h,\delta}[\Pi_h v] = A[v]$.
\end{enumerate}
\end{lem}

\begin{proof}
We first show (i).
Let $v \in X^2$ and $v_{h,\delta} \in V_h$ with $v_{h,\delta} \to v$ in $H^1_\pi$.
Set
\begin{equation}
\chi(x) = \begin{cases}
1, & \text{if } v(x) = \psi(x), \\
0, & \text{otherwise},
\end{cases}
\qquad
\chi_{h,\delta}(x) = \sum_{j=1}^{N} \zeta_{\delta,j-1} \zeta_{\delta,j} \chi_{I_j}(x),
\end{equation}
where $\chi_{I_j}$ is the characteristic function of $I_j$.
Then,
\begin{equation}
-A_{h,\delta}[v_{h,\delta}] + A[v]
= J_1 + J_2,
\label{eq:A_h,delta}
\end{equation}
where
\begin{equation}
J_1
= - \int_0^1 \chi_{h,\delta} \left( \sqrt{1+|v'_{h,\delta}|^2} - \sqrt{1+|v'|^2} \right) dx,
\quad
J_2
= - \int_0^1 (\chi_{h,\delta} - \chi) \sqrt{1+|v'|^2} dx.
\end{equation}
Noting that $|\chi_{h,\delta}| \le 1$ almost everywhere, we can bound $J_1$ as
\begin{equation}
|J_1| \le \int_0^1 |v'_{h,\delta} - v'| dx
\le \| v_{h,\delta} - v\|_{H^1}
\end{equation}
owing to \eqref{eq:sqrt}, which implies
\begin{equation}
\lim_{h \downarrow 0} J_1 = 0.
\label{eq:J1}
\end{equation}
Next, we show
\begin{equation}
\liminf_{h \downarrow 0} J_2 \ge 0.
\label{eq:J2}
\end{equation}
Assume that $x \in \{ v=\psi \}$.
It is clear that $\chi(x) - \chi_{h,\delta}(x) \ge 0$.
On the other hand, if $x \in \{ v>\psi \}$, then we can show that
\begin{equation}
\lim_{h \downarrow 0} (\chi(x) - \chi_{h,\delta}(x)) = 0
\label{eq:chi}
\end{equation}
since $v_{h,\delta} \to v$ in $H^1_\pi$ and $H^1(0,1)$ is continuously embedded in $C^0[0,1]$.
Thus, Fatou's lemma yields the estimate \eqref{eq:J2}.
The assertion \eqref{eq:lsc-A} is a consequence of \eqref{eq:A_h,delta}, \eqref{eq:J1}, and \eqref{eq:J2}.

Let us prove the assertion (ii).
Since $\Pi_h v \to v$ in $H^1_\pi$ for $v \in X^2$, it suffices to show
\begin{equation}
\lim_{h \downarrow 0} J_2 = 0
\label{eq:J2-2}
\end{equation}
for $v_{h,\delta} = \Pi_h v$.
We prove that \eqref{eq:chi} holds for $x \in \operatorname{Int}\{ u=\psi \}$.
Let $x \in \operatorname{Int}\{ u=\psi \}$.
For each sufficiently small $h$, we can find an index $j$ such that $x \in I_j \in \operatorname{Int}\{ u=\psi \}$.
For such $j$, we have $(u-\psi)(x_{j-1}) = (u-\psi)(x_j) = 0$, which implies $\chi_{h,\delta}(x) = \chi(x) = 1$.
Thus, we have \eqref{eq:chi}.
Noting that $u - \psi$ is continuous in $[0,1]$, we can obtain \eqref{eq:chi} for almost all $x \in (0,1)$, which yields \eqref{eq:J2-2}.
\end{proof}

\begin{lem}\label{lem:penalty}
Define a functional $P$ on $H^1_\pi$ as
\begin{equation}
P[v] = \begin{cases}
0, & v \in X^1, \\ +\infty, & \text{otherwise}.
\end{cases}
\end{equation}
Then, the following two assertions hold.
\begin{enumerate}
\item For $v \in H^1_\pi$,
\begin{equation}
\limsup_{h, \rho \downarrow 0} P_{h,\rho}[\Pi_h v] \le P[v].
\end{equation}
\item For $v \in H^1_\pi$ and for $v_{h,\rho} \in H^1_\pi$ with $v_{h,\rho} \to v$ in $H^1_\pi$,
\begin{equation}
\liminf_{h, \rho \downarrow 0} P_{h,\rho}[v_{h,\rho}] \ge P[v].
\end{equation}
\end{enumerate}
\end{lem}

\begin{proof}
The assertion (i) is obvious since
\begin{align}
v \in X^1 & \implies P_{h,\rho}[\Pi_h v] = P[v] = 0, \\
v \not\in X^1 & \implies P_{h,\rho}[\Pi_h v] \le +\infty = P[v].
\end{align}
We show (ii).
We can assume $v \not\in X^1$ and $v_{h,\rho} \in V_h$. Let $d := \sup_{x \in [0,1]} (\psi(x) - v(x))$, and
\begin{equation}
J_d := \{ x \in [0,1] \mid v(x) + d/2 < \psi(x) \}.
\end{equation}
Note that $d>0$ and $|J_d| > 0$.
Since $v_{h,\delta} \in V_h$ in $H^1_\pi$ and  $H^1(0,1)$ is continuously embedded in $C^0[0,1]$, we can show that
\begin{equation}
J_d \subset \{ x \in [0,1] \mid v_{h,\delta}(x) + d/4 < \psi(x) \}
\end{equation}
for sufficiently small $h$ and $\rho$.
Thus,
\begin{equation}
P_{h,\rho}[v_{h,\rho}]
\ge \frac{1}{\rho} \sum_{x_j \in J_d} |\psi(x_j) - v_{h,\rho}(x_j)|^2 h
\ge \frac{1}{\rho} \sum_{x_j \in J_d} \left( \frac{d}{4} \right)^2 h
\ge \frac{1}{\rho} \frac{d^2}{16} (|J_d| - 2h)
\to +\infty
\end{equation}
as $h, \rho \downarrow 0$, which implies the assertion.
\end{proof}

\subsection{Bending energy}
The main difficulty in the proof of Theorem \ref{thm:gamma} is the estimate of the bending energy $B$.
In this subsection, we establish the two inequalities (U) and (L) for $B$.
We redefine $B$ as a functional on $H^1_\pi$ as follows:
\begin{equation}
B[v] = \begin{dcases}
\frac{C}{2} \int_0^1 \frac{|v''|^2}{(1+|v'|^2)^{2/5}} dx, & v \in H^2_\pi, \\
+ \infty , & H^1_\pi \setminus H^2_\pi.
\end{dcases}
\end{equation}
Moreover, we introduce an auxiliary functional $\tilde{B}_{h}$ by
\begin{equation}
\tilde{B}_{h}[v] := \begin{dcases}
\frac{C}{2} \sum_{j=1}^N \frac{|D_{h,j}|^2}{(1+|d_{h,j}v|^2)^{5/2}}h, & v \in V_h, \\
+ \infty, & v \in H^1_\pi \setminus V_h.
\end{dcases}
\end{equation}

\begin{lem}\label{lem:bending-1}
For every $v \in H^1_\pi$,
\begin{equation}
\lim_{h \downarrow 0} \tilde{B}_h[\Pi_h v] = B[v].
\label{eq:bending-1}
\end{equation}
\end{lem}

\begin{proof}
We can assume $v \in H^2_\pi$.
We first show that \eqref{eq:bending-1} holds for $v \in C^\infty[0,1]$.
Let $v \in C^\infty[0,1]$.
Then,
\begin{equation}
\tilde{B}_h[\Pi_h v] - B[v] = B_1 + B_2,
\end{equation}
where
\begin{equation}
B_1 = \sum_{j=1}^N \int_{I_j} \frac{|D_{h,j}v|^2 - |v''|^2}{(1+|d_{h,j}v|^2)^{5/2}} dx, \quad
B_2 = \sum_{j=1}^N \int_{I_j} |v''|^2 \left[ \frac{1}{(1+|d_{h,j}v|^2)^{5/2}} - \frac{1}{(1+|v'|^2)^{5/2}} \right] dx.
\end{equation}
By the Taylor expansion and \eqref{eq:D_j},
\begin{align}
|D_{h,j}v - v''(x)|
&\le |D_{h,j}v - v''(x_j)| + |v''(x_j) - v''(x)|
\le \frac{h^2}{12} \| v^{(4)} \|_{L^\infty(0,1)} + h \| v^{(3)} \|_{L^\infty(0,1)}, \\
|D_{h,j}v|
& \le \frac{1}{h^2} \| \psi \|_{L^1(x_{j-1},x_{j+1})} \| v'' \|_{L^\infty(0,1)}
= \| v'' \|_{L^\infty(0,1)}
\end{align}
Thus,
\begin{align}
|B_1|
&\le \sum_{j=1}^N \int_{I_j} |D_{h,j}v + v''| |D_{h,j}v - v'' | dx \\
&\le \sum_{j=1}^N 2 \| v'' \|_{L^\infty(0,1)}
\left( \frac{h^2}{12} \| v^{(4)} \|_{L^\infty(0,1)} + h \| v^{(3)} \|_{L^\infty(0,1)} \right) h \\
&\to 0
\end{align}
as $h \downarrow 0$.
Noting that the function $(1+x^2)^{-5/2}$ is Lipschitz continuous with the estimate
\begin{equation}
| (1+s^2)^{-5/2} - (1+t^2)^{-5/2} | \le 5|s-t|,
\label{eq:sqrt-5/2}
\end{equation}
we have
\begin{align}
|B_2| &\le 5 \sum_{j=1}^N \int_{I_j} |v''|^2 |d_{h,j}v - v'| dx \\
& \le 5 \sum_{j=1}^N \int_{I_j} \|v''\|_{L^\infty(0,1)}^2 h \|v''\|_{L^\infty(0,1)} dx \\
& = 5h\|v''\|_{L^\infty(0,1)}^3 \\
& \to 0
\end{align}
as $h \downarrow 0$, as a result of Lemma \ref{lem:fdm}.
Therefore, we obtain \eqref{eq:bending-1} for $v \in C^\infty[0,1]$.

Now, we establish our assertion.
Let $v \in H^2_\pi$.
Then, there exists $v_n \in C^\infty[0,1]$ satisfying $v_n \to v$ in $H^2(0,1)$.
We write
\begin{equation}
d_j := d_j v, \quad d_{j,n}:=d_{h,j}v_n, \quad D_j := D_{h,j}v, \quad D_{j,n} := D_{h,j}v_n.
\end{equation}
Then,
\begin{equation}
\tilde{B}_h[\Pi_h v] - \tilde{B}_h[\Pi_h v_n] = B_3^n + B_4^n,
\quad
B[v_n] - B[v] = B_5^n + B_6^n,
\end{equation}
where
\begin{gather}
B_3^n = \sum_{j=1}^N \frac{D_j^2 - D_{j,n}^2}{(1+d_{j,n}^2)^{5/2}} h, \quad
B_4^n = \sum_{j=1}^N D_j^2 \left[ \frac{1}{(1+d_j^2)^{5/2}} - \frac{1}{(1+d_{j,n}^2)^{5/2}} \right] h, \\
B_5^n = \int_0^1 \frac{|v''_n|^2 - |v''|^2}{(1+|v'_n|^2)^{5/2}} dx, \quad
B_6^n = \int_0^1 |v''|^2 \left[ \frac{1}{(1+|v'_n|^2)^{5/2}} - \frac{1}{(1+|v'|^2)^{5/2}} \right] dx.
\end{gather}
As a result of Lemma \ref{lem:fdm}, \eqref{eq:sqrt-5/2}, the H\"oder inequality, and the Sobolev inequality, we have
\begin{align}
|B_3^n| &\le \sum_{j=1}^N |D_j + D_{j,n}| |D_j - D_{j,n}| h
\le C \| v'' + v''_n \|_{L^{2}(0,1)} \| v'' - v''_n \|_{L^2(0,1)}
\to 0 \\
|B_4^n| &\le 5 \sum_{j=1}^N |d_j - d_{j,n}| |D_j|^2 h
\le C \| v' - v'_n \|_{L^\infty(0,1)} \|v''\|_{L^2(0,1)} \to 0 \\
|B_5^n| &\le C \| v'' + v''_n \|_{L^{2}(0,1)} \| v'' - v''_n \|_{L^2(0,1)}
\to 0 \\
|B_6^n| &\le C \| v' - v'_n \|_{L^\infty(0,1)} \|v''\|_{L^2(0,1)} \to 0
\end{align}
as $n \to \infty$, where $C>0$ is independent of $h$, $v$, and $v_n$.
These estimates yield
\begin{equation}
\Big| \tilde{B}_h[\Pi_h v] - \tilde{B}_h[\Pi_h v_n] \Big| \to 0
,\quad \Big| B[v_n] - B[v] \Big| \to 0
\end{equation}
as $n \to \infty$.
Since
\begin{equation}
\Big| \tilde{B}_h[\Pi_h v] - B[v] \Big|
\le \Big| \tilde{B}_h[\Pi_h v] - \tilde{B}_h[\Pi_h v_n] \Big|
+ \Big| \tilde{B}_h[\Pi_h v_n] - B[v_n] \Big|
+ \Big| B[v_n] - B[v] \Big|
\end{equation}
and $v_n \in C^\infty[0,1]$, we can complete the proof.
\end{proof}

Now, we can establish condition (U) for $B$.
\begin{lem}\label{lem:bending-2}
For every $v \in H^1_\pi$,
\begin{equation}
\limsup_{h\downarrow 0} B_h[\Pi_h v] \le B[v].
\end{equation}
\end{lem}

\begin{proof}
We can assume $v \in H^2_\pi$.
It suffices to show
\begin{equation}
\lim_{h \downarrow 0} \left( \tilde{B}_h[\Pi_h v] - B_h[\Pi_h v] \right) = 0
\label{eq:bending-5}
\end{equation}
as a result of Lemma \ref{lem:bending-1}.
In the following, we fix small $\varepsilon > 0$ arbitrarily, and we use the same notation as in \eqref{eq:discretize} for $v_h = \Pi_h v$

We first show that there exists $h_\varepsilon > 0$ satisfying
\begin{equation}\max_j | d_{j+1} - d_j| < \varepsilon,
\quad
\theta_h := \max_j | \theta_j | < 2\varepsilon.
\label{eq:bending-2}
\end{equation}
From \eqref{eq:D_j},
\begin{equation}
|d_{j+1} - d_j| = \frac{|D_j|}{h} \le C_0 \| v'' \|_{L^2(x_{j-1},x_{j+1})},
\end{equation}
where $C_0 > 0$ is independent of $j$, $h$, and $v$.
Since $v'' \in L^2(0,1)$, there exists $\delta>0$ such that
\begin{equation}
\| v'' \|_{L^2(K)} < \varepsilon / C_0
\end{equation}
for all measurable sets $K \subset (0,1)$ with $\operatorname{meas}(K) < \delta$.
Thus,
\begin{equation}
|d_{j+1} - d_j| < \varepsilon
\end{equation}
for $h < \delta/2$.
Moreover, one can check that
\begin{equation}
1 + d_j d_{j+1} = \left( \frac{d_{j+1} + d_j}{2} \right)^2 + 1 - \left( \frac{d_{j+1} - d_j}{2} \right)^2 \ge 1 - \frac{\varepsilon^2}{4} > \frac{1}{2}
\label{eq:bending-3}
\end{equation}
for $\varepsilon < \sqrt{2}$, and thus,
\begin{equation}
|\theta_j| = \arctan \frac{|d_{j+1} - d_j|}{1+d_j d_{j+1}} < \arctan 2\varepsilon \le 2\varepsilon.
\end{equation}

Now, we show that
\begin{equation}
\left| \tilde{B}_h[\Pi_h v] - B_h[\Pi_h v] \right|
\le C (\varepsilon^2 + \varepsilon) \|v\|_{H^2(0,1)},
\qquad
\forall h \le h_\varepsilon, \quad
\forall \varepsilon < \sqrt{2}
\label{eq:bending-4}
\end{equation}
for some $C>0$ independent of $\varepsilon$ and $h$, which implies \eqref{eq:bending-5}.
By simple calculation,
\begin{equation}
\tilde{B}_h[\Pi_h v] - B_h[\Pi_h v]
= B_1 + B_2,
\end{equation}
where
\begin{align}
B_1 &= \sum_{j=1}^N
\left[ 1 - \left( \frac{\theta_j}{\tan\theta_j} \right)^2 \right]
\frac{\tan^2 \theta_j}{l_j l_{j+1}}
\frac{l_j^3 + l_{j+1}^3}{(l_j + l_{j+1})^2},\\
B_2 &= \sum_{j=1}^N
\frac{D_j^2 h}{(1+d_j^2)^{5/2}}
\left[ 1 - \frac{2(1+d_j^2)[ (1+d_j^2)^{3/2} + (1+d_{j+1}^2)^{3/2} ]}{(1+d_jd_{j+1})^2 (1+d_{j+1}^2)^{1/2} [ (1+d_j^2)^{1/2} + (1+d_{j+1}^2)^{1/2} ]^2 } \right]
\end{align}
We first estimate $B_1$.
From the Sobolev inequality,
\begin{equation}
|d_j| \le \| v'\|_{L^\infty(0,1)} \le C \|v\|_{H^2(0,1)}
\end{equation}
for some $C>0$.
Thus, equations \eqref{eq:bending-2} and \eqref{eq:bending-3} yield
\begin{gather}
1 - \left( \frac{\theta_j}{\tan\theta_j} \right)^2
\le C \theta_j^2 \le C \varepsilon^2, \\
\frac{\tan^2 \theta_j}{l_j l_{j+1}}
= \frac{4D_j^2}{\sqrt{1+d_j^2}\sqrt{1+d_{j+1}^2}(1+d_jd_{j+1})^2} \le 16 D_j^2, \\
\frac{l_j^3 + l_{j+1}^3}{(l_j + l_{j+1})^2}
= \frac{h [ (1+d_j^2)^{3/2} + (1+d_{j+1}^2)^{3/2} ]}{2(\sqrt{1+d_j^2} + \sqrt{1+d_{j+1}^2})}
\le C(1 + \|v\|_{H^2(0,1)}^2)^{3/2} h,
\end{gather}
which implies
\begin{equation}
|B_1| \le C \varepsilon^2 \sum_{j=1}^N D_j^2 h \le C \varepsilon^2 \| v\|_{H^2(0,1)}^2
\end{equation}
from Lemma \ref{lem:fdm}.
Next, we consider $B_2$.
Let
\begin{equation}
g(s,t) = \frac{2(1+s^2)[(1+s^2)^{3/2} + (1+t^2)^{3/2}]}{(1+st)^2(1+t^2)^{1/2}[(1+s^2)^{1/2} + (1+t^2)^{1/2}]},
\end{equation}
for $s,t \in \bR$.
Then,
\begin{equation}
|1 - g(s,t)| \le C \varepsilon
\end{equation}
if $|s-t| < \varepsilon$ for some $C>0$.
Thus,
\begin{equation}
|B_2| \le \sum_{j=1}^N \frac{D_j^2 h}{(1+d_j^2)^{5/2}}
| 1- g(d_j, d_{j+1})| \le C \varepsilon \| v\|_{H^2(0,1)},
\end{equation}
from Lemma \ref{lem:fdm}.
Hence, we obtain \eqref{eq:bending-4} and can complete the proof.
\end{proof}

Next, we focus on condition (L).
\begin{lem}\label{lem:bending-3}
If $v \in H^2_\pi$ and $v_n \in H^2_\pi$ satisfy $v_n \to v$ in $H^1_\pi$, then
\begin{equation}
B[v] \le \liminf_{n \to \infty} B[v_n].
\end{equation}
\end{lem}

\begin{proof}
Let $G(t) = \int_{-\infty}^{t} (1+s^2)^{-5/4} ds$.
Then, $G$ satisfies
\begin{equation}
|G(s) - G(t)| \le |s - t|, \quad \forall s,t \in \bR,
\end{equation}
since $|G'(t)| = (1+t^2)^{-5/4} \le 1$.
Thus, $G(v'_n) \to G(v)$ in $L^2(0,1)$.
Since
\begin{equation}
\frac{d}{dx} G(v') = \frac{v''}{(1+|v'|^2)^{4/5}},
\end{equation}
it follows that
\begin{align}
\left| \left( \frac{v''}{(1+|v'|^2)^{4/5}} - \frac{v''_n}{(1+|v'_n|^2)^{4/5}},\ \varphi \right)_{L^2(0,1)} \right|
&= | (G(v') - G(v'_n), \ \varphi')_{L^2(0,1)}| \\
& \le \| (G(v') - G(v'_n) \|_{L^2(0,1)} \| \varphi' \|_{L^2(0,1)} \\
& \to 0
\end{align}
as $n \to \infty$, for every $\varphi \in C^\infty_c(0,1)$.
This yields the desired assertion since $B[v] = \| G(v')\|_{L^2(0,1)}^2$.
\end{proof}

Now, we are ready to obtain condition (L) for $B$.
Note that the following estimate is not uniform with respect to $S$.
The same estimate was obtained in another topology by \cite{BruNR01}.
\begin{lem}\label{lem:bending-4}
Let $S>0$.
For each $v \in X_S$ and $v_h \in X_S$ with $v_h \to v$ in $H^1_\pi$,
\begin{equation}
B[v] \le \liminf_{h \downarrow 0} B_h[v_h].
\label{eq:bending-6}
\end{equation}
\end{lem}

\begin{proof}
Let $v \in X_S$ and $v_h \in X_S$ with $v_h \to v$ in $H^1_\pi$.
In general, we can assume $\sup_h B_h[v_h] < +\infty$, i.e., $v_h \in V_h$.
We consider the following three cases:
\begin{description}
\item[Case 1.] $v \in H^2_\pi$.
\item[Case 2.] $v \in (X_S \cap H^2(0,1)) \setminus H^2_\pi$.
\item[Case 3.] $v \in X_S \setminus H^2(0,1)$.
\end{description}
\begin{figure}[htb]
\centering
\includegraphics{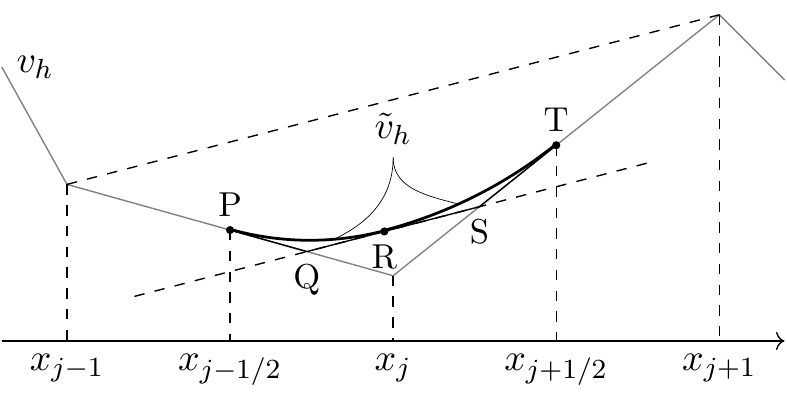}
~~~~~
\includegraphics{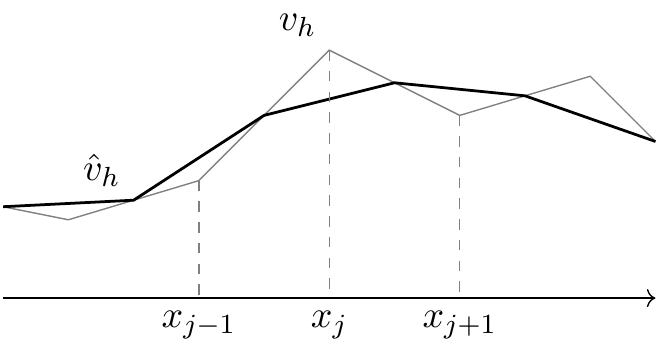}
\caption{Definition of $\tilde{v}_h$ and $\hat{v}_h$.  In the left figure, the segment $\overline{QS}$ is parallel to $\overline{PT}$. Moreover, $|PQ|=|QR|$ and $|RS|=|ST|$, where $|XY|$ is the length of the segment $\overline{XY}$. The bold lines are two circular arcs. The left arc is tangential to the segments $\overline{PQ}$ and $\overline{QR}$ at $P$ and $R$, respectively. The right arc is as well.}
\label{fig:piecewise-arc}
\end{figure}

\textbf{Case 1}. Assume $v \in H^2_\pi$.
Let $\tilde{v}_h \in H^2_\pi$ be the piecewise-circular arc function constructed in Construction 1 of \cite[Section 4]{BruNR01}.
For the reader's convenience, we provide an illustration of $\tilde{v}_h$ in Figure \ref{fig:piecewise-arc}.
Then, it is shown by \cite{BruNR01} that
\begin{equation}
B[\tilde{v}_h] = B_h[v_h](1 + O(\theta_h^2)),
\label{eq:bending-7}
\end{equation}
when $\theta_h := \max_j |\theta_j(v_h)|$ is small.
We show that $v_h - \tilde{v}_h \to 0$ in $H^1_\pi$.
Let $\hat{v}_h \in H^1_\pi$ be the polygonal curve whose vertices are $(x_{j-1/2}, v_{j-1/2})$, $j=1,\dots,N$, where $x_{j-1/2} = (j-1/2)h$ and $v_{j-1/2} = (v_{j-1} + v_j)/2$ (see Figure \ref{fig:piecewise-arc}), and let $J_j = (x_{j-1/2}, x_{j+1/2})$.
Then, from the convexity or concavity of $\tilde{v}_h|_{J_j}$,
\begin{equation}
|v_h - \tilde{v}_h| \le |v_h - \hat{v}_h|
\le \frac{h}{2} \left| \frac{d_{h,j}(v_h)}{2} + \frac{d_{h,j+1}(v_h)}{2} \right|,
\quad \text{on } J_j
\end{equation}
for each $j$.
Thus,
\begin{equation}
\| v_h - \tilde{v}_h \|_{L^2(I_j)}^2
\le \frac{h^2}{8} \|v'_h\|_{L^2(x_{j-1},x_{j+1})}^2
\end{equation}
by Lemma \ref{lem:fdm}, which implies
\begin{equation}
\| v_h - \tilde{v}_h \|_{L^2(0,1)}
\le \frac{h}{2} \|v'_h\|_{L^2(0,1)}
\to 0
\end{equation}
as $h \downarrow 0$.
Again, by the convexity or concavity of $\tilde{v}_h|_{J_j}$,
\begin{equation}
\min\{ d_{h,j}(v_h), d_{h,j+1}(v_h) \}
\le \tilde{v}'_h|_{J_j}
\le \max\{ d_{h,j}(v_h), d_{h,j+1}(v_h) \}.
\end{equation}
Therefore,
\begin{equation}
| v'_h - \tilde{v}'_h |
\le |d_{h,j+1}(v_h) - d_{h,j}(v_h)|
\le |d_{h,j+1}(v_h - v)| + |d_{h,j}(v_h - v)| + h|D_{h,j}(v)|,
\end{equation}
which implies
\begin{equation}
\| v'_h - \tilde{v}'_h \|_{L^2(0,1)}
\le C( \|v' - v'_h \|_{L^2(0,1)} + h\|v''\|_{L^2(0,1)} )
\to 0
\end{equation}
as $h \downarrow 0$, due to \eqref{eq:lem-fdm-1} and \eqref{eq:lem-fdm-3}.
Hence, we obtain $v_h - \tilde{v}_h \to 0$ in $H^1_\pi$.

Now, we are ready to show \eqref{eq:bending-6}.
Noting that
\begin{equation}
l_h := \max_j l_j(v_h) \le \frac{1}{2}(h + \sqrt{h}\|v'_h\|_{L^2(0,1)}) \to 0
\label{eq:bending-8}
\end{equation}
as $h \downarrow 0$ by \eqref{eq:lem-fdm-1} and that
\begin{equation}
B_h[v_h]
\ge \frac{C}{2} \frac{\theta_j^2}{2l_h}
\label{eq:bending-9}
\end{equation}
for each $j$ by \eqref{eq:bending-0}, we can assume $\theta_h \to 0$ as $h \downarrow 0$.
Then,
\begin{equation}
\lim_{h \downarrow 0} \frac{B[\tilde{v}_h]}{B_h[v_h]} \to 1
\end{equation}
from \eqref{eq:bending-7}, and
\begin{equation}
B[v] \le \liminf_{h \downarrow 0} B[\tilde{v}_h]
\end{equation}
from Lemma \ref{lem:bending-3}.
Hence, we can obtain \eqref{eq:bending-6} for $v \in H^2_\pi$.

\textbf{Case 2.} Assume $v \in (X_S \cap H^2(0,1)) \setminus H^2_\pi$.
Then, we can find $\bar{\theta} > 0$ such that
\begin{equation}
|\theta_N| \ge \bar{\theta}
\end{equation}
for every $h$, since $v'(0) \ne v'(1)$.
Thus, \eqref{eq:bending-8} and \eqref{eq:bending-9} implies
\begin{equation}
\liminf_h B_h[v_h] = \infty = B[v].
\label{eq:bending-11}
\end{equation}

\textbf{Case 3.} Assume $v \in X_S \setminus H^2(0,1)$.
Note that $v_h \in X_S$, and thus, $l_j(v_h) \le \frac{h}{2}\sqrt{1+S^2}$.
Since $\arccos t \ge \sqrt{1-t^2}$ for $t \in (-1,1)$,
\begin{equation}
\theta_j^2
\ge \frac{|d_{h,j+1}(v_h) - d_{h,j}(v_h)|^2}{(1+d_{h,j}(v_h)^2)(1+d_{h,j+1}(v_h)^2)}
\ge \frac{h^2 |D_{h,j}(v_h)|^2}{(1+S^2)^2}
\end{equation}
Therefore,
\begin{equation}
B[v_h] \ge \frac{C}{2} \frac{1}{(1+S^2)^{5/2}} \sum_{j=1}^N |D_{h,j}(v_h)|^2 h
\label{eq:bending-10}
\end{equation}
for $v_h \in X_S$.
Since $v_h$ does not converge to an element in $H^2$, we can obtain
\begin{equation}
\lim_{h\downarrow 0} \sum_{j=1}^N |D_{h,j}(v_h)|^2 h = + \infty
\end{equation}
by contradiction (cf. the characterization of $H^1$-functions by shift operators).
Hence, we have \eqref{eq:bending-11}, and thus, the proof is completed.
\end{proof}

\subsection{Completion of the proof of Theorem \ref{thm:gamma}}
Now, we are ready to prove the main theorem of this section.

\begin{proof}[Proof of Theorem \ref{thm:gamma}]
The condition (U) is a consequence of Lemmas \ref{lem:tension}, \ref{lem:adhesion}, \ref{lem:penalty}, and \ref{lem:bending-2}.
We establish the condition (L).
Let $v \in X_S$, $v_{h,\delta,\rho} \in X_S$ and $v_{h,\delta,\rho} \to v$ in $H^1_\pi$.
We can assume $v_{h,\delta,\rho} \in V_h$.
If $v \in X^2$, then condition (L) can be obtained from Lemmas \ref{lem:tension}, \ref{lem:adhesion}, \ref{lem:penalty}, and \ref{lem:bending-4}.
Let $v \in X_S \setminus X^2$.
Then,
\begin{align}
v \not \ge \psi & \implies P_{h,\rho}[v_{h,\delta,\rho}] \to +\infty, \\
v \not\in H^2_\pi & \implies B_h[v_{h,\delta,\rho}] \to +\infty
\end{align}
as $h,\delta,\rho \downarrow 0$, as a result of Lemmas \ref{lem:penalty} and \ref{lem:bending-4}.
Since the estimate \eqref{eq:sufficient-2} implies
\begin{equation}
E_{h,\delta,\rho}[v_{h,\delta,\rho}]
\ge B_h[v_{h,\delta,\rho}] + P_{h,\rho}[v_{h,\delta,\rho}] - 4\gamma (T[\psi]/\sigma + c_0),
\label{eq:bending-12}
\end{equation}
we obtain
\begin{equation}
\lim_{h,\delta,\rho \downarrow 0}E_{h,\delta,\rho}[v_{h,\delta,\rho}]
= +\infty = E[v]
\end{equation}
for $v \in X_S \setminus X^2$.
Hence, we can complete the proof.
\end{proof}

\section{Compactness and convergence results}
\label{section:compactness}
One of our aims is to show that a sequence of discrete minimizers converges to a continuous minimizer.
From this viewpoint, the fundamental theorem of $\Gamma$-convergence (Lemma~\ref{lem:fundamental}) plays an important role.
According to Lemma~\ref{lem:fundamental}, what remains to be shown is that a sequence of minimizers of $E_{h,\delta,\rho}$ has a cluster point.
For this purpose, we show the discrete version of the compact embedding $H^2 \hookrightarrow H^1$.
In what follows, we use the following notation:
\begin{equation}
\| v_h \|_{h,0,p} := \left( \sum_{j=1}^N |v_h(x_j)|^p h \right)^{1/p},
\quad
| v_h |_{h,2,p} := \left( \sum_{j=1}^N |D_{h,j}(v_h)|^p h \right)^{1/p}
\end{equation}
for $v_h \in V_h$ and $p \in [1,\infty)$. Note that $\| \cdot \|_{h,0,p}$ is an equivalent norm to $\| \cdot \|_{L^p(0,1)}$ in $V_h$ (cf.~\cite{FujSS01}).

\begin{lem}[Discrete Rellich-type theorem]
\label{lem:disc-Rellich}
Let $p \in [1,\infty)$.
Assume that the sequence $\{ v_h \mid v_h \in V_h \}_h$ satisfies
\begin{equation}
\| v_h \|_{W^{1,p}(0,1)} + | v_h |_{h,2,p} \le M
\end{equation}
uniformly for $h > 0$.
Then, we can find a subsequence $\{ v_{h'} \}_{h'}$ and $\bar{v} \in W^{1,p}_\pi$ that satisfies
\begin{equation}
v_{h'} \to \bar{v} \quad \text{in } W^{1,p}_\pi
\end{equation}
as $h' \downarrow 0$, where $W^{1,p}_\pi = \{ v \in W^{1,p}(0,1) \mid v(0) = v(1) \}$.
\end{lem}

\begin{proof}
Since $\| v_h \|_{W^{1,p}(0,1)}$ is bounded and the embedding $W^{1,p}(0,1) \hookrightarrow L^p(0,1)$ is compact, we can find a subsequence $\{ v_{h'} \}_{h'}$ and $\bar{v} \in L^p(0,1)$ that satisfies
$v_{h'} \to \bar{v}$ in $L^p(0,1)$ as $h' \downarrow 0$.
We show that $\bar{v} \in W^{1,p}_\pi$ and that there exists a subsequence $\{ v_{h''} \}_{h''} \subset \{ v_{h'} \}_{h'}$ such that $v_{h''} \to \bar{v}'$ in $L^p(0,1)$ as $h'' \downarrow 0$.
\begin{figure}[htb]
\centering
\includegraphics{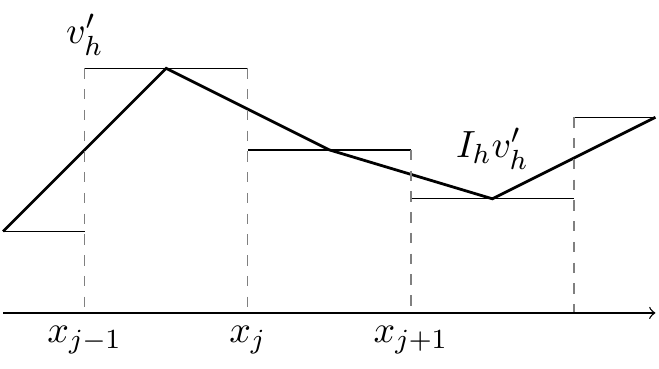}
\caption{Definition of $I_h v'_h$.}
\label{fig:interpolation}
\end{figure}

For $v_h \in V_h$, we construct a ``continuous version'' of $v'_h$ as
\begin{equation}
(I_h v'_h)|_{(x_{j-1/2},x_{j+1/2})}(x)
= d_{h,j}(v_h)\frac{x-x_{j-1/2}}{h} + d_{h,j+1}(v_h)\frac{x_{j+1/2}-x}{h}
\end{equation}
for $j=1,\dots,N$ (see figure \ref{fig:interpolation}).
Note that
\begin{equation}
\| I_h v'_h \|_{h,0,p} = |v_h|_{W^{1,p}(0,1)}, \quad
\| (I_h v'_h)' \|_{L^p(0,1)} = |v_h|_{h,2,p}.
\end{equation}
Thus, there exists $M'>0$ such that $\| I_h v'_h \|_{W^{1,p}(0,1)} \le M'$ for every $h>0$.
Therefore, we can extract a subsequence $\{ v_{h''} \}_{h''} \subset \{ v_{h'} \}_{h'}$ and $\bar{w} \in L^p(0,1)$ that satisfies $I_h v'_{h''} \to \bar{w}$ in $L^p(0,1)$ as $h'' \downarrow 0$.
One can check that
\begin{equation}
\| v'_h - I_h v'_h \|_{L^p(0,1)}
\le (1+p)^{-1/p} |v_h|_{h,2,p}h
\le (1+p)^{-1/p} M h \to 0
\end{equation}
as $h \downarrow 0$.
Thus, $v'_{h''} \to \bar{w}$ in $L^p(0,1)$, which implies $\bar{v} \in W^{1^p}_\pi$ and $v_{h''} \to \bar{v}$ in $W^{1,p}_\pi$.
This is the desired assertion.
\end{proof}

Now, we can establish the convergence result and the (local) optimization of the original problem \eqref{eq:cont-problem-S}.
\begin{thm}\label{thm:compactness-S}
Let $\bar{v}^{(S)}_h = \bar{v}^{(S)}_{h,\delta,\rho} \in V_h$ be a minimizer of the problem \eqref{eq:disc-problem-S}.
Then, for every $S>0$, the sequence $\{ \bar{v}^{(S)}_h \}_h$ has a cluster point $\bar{v}^{(S)} \in \overline{X_S}$, where $\overline{X_S}$ is the closure of $X_S$ in the topology of $H^1_\pi$.

Therefore, the problem \eqref{eq:cont-problem-S} has at least one solution $\bar{v}^{(S)}$, and the sequence $\{ \bar{v}^{(S)}_h \}_h$ has a subsequence that converges to $\bar{v}^{(S)}$ in $H^1_\pi$.
\end{thm}

\begin{proof}
In view of Lemma \ref{lem:disc-Rellich}, it suffices to show
\begin{equation}
\| \bar{v}^{(S)}_h \|_{H^1(0,1)} + | \bar{v}^{(S)}_h |_{h,2,2} \le M
\label{eq:boundedness}
\end{equation}
for some $M>0$.
Let $c_h \in V_h$ be a constant function with $c_h \ge \max \psi + \delta$.
Then,
\begin{equation}
E_{h,\delta,\rho}[\bar{v}^{(S)}_{h}] \le E_{h,\delta,\rho}[c_h] = \sigma.
\label{eq:boundedness-1}
\end{equation}
Thus, \eqref{eq:bending-10} and \eqref{eq:bending-12} yield
\begin{equation}
|\bar{v}^{(S)}_h|_{h,2,2}^2 \le \frac{2}{C}(1+S^2)^{5/2}
\left[ \sigma + 4\gamma (T[\psi]/\sigma + c_0) \right].
\end{equation}
Moreover, $\bar{v}^{(S)}_h \in X_S$ implies
\begin{equation}
\| \bar{v}^{(S)}_h \|_{H^1(0,1)} \le \| \bar{v}^{(S)}_h \|_{W^{1,\infty}(0,1)} \le S.
\end{equation}
Hence, we obtain \eqref{eq:boundedness}, and we can complete the proof.
\end{proof}

\begin{rem}
From Lemma~\ref{lem:disc-Rellich}, we can obtain only $\bar{v}^{(S)} \in \overline{X_S}$, and there still remains a possibility that $\bar{v}^{(S)} \not\in X^2$.
However, the fundamental theorem of $\Gamma$-convergence (Lemma~\ref{lem:fundamental}) guarantees that $\bar{v}^{(S)} \in X^2$.
\end{rem}

Theorem \ref{thm:compactness-S} says that the problem \eqref{eq:problem0} has a global solution provided that the discrete minimizer of the problem \eqref{eq:disc-problem-S} is Lipschitz continuous uniformly with respect to $S$.
We provide a sufficient condition for this property.
For example, the condition \eqref{eq:sufficient} holds when $C$ is sufficiently large.

\begin{thm}[Sufficient condition for global optimization]
\label{thm:sufficient}
Let $\bar{v}^{(S)}_h = \bar{v}^{(S)}_{h,\delta,\rho} \in V_h$ be a minimizer of the problem \eqref{eq:disc-problem-S}.
Assume that there exists $c_0 > 0$ such that $\delta \le c_0 h$. Moreover, suppose that the physical parameters satisfy
\begin{equation}
\frac{1}{\sqrt{2C\sigma}} \left[ \sigma + 4\gamma \left( \frac{T[\psi]}{\sigma} + c_0 \right) \right]
+ \arctan ( |\psi|_{W^{1,\infty}(0,1)} + 2c_0)
\le \bar{\phi}
\label{eq:sufficient}
\end{equation}
for some $\bar{\phi} \in (0, \pi/2)$, which is independent of $h$, $\delta$, $\rho$, and $S$.
Then, the sequence $\{\bar{v}^{(S)}_h\}_h$ satisfies
\begin{equation}
| \bar{v}^{(S)}_h |_{W^{1,\infty}(0,1)} \le \tan\bar{\phi}
\label{eq:sufficient-0}
\end{equation}
for all $S>0$.

Therefore, the problem \eqref{eq:cont-problem} has at least one solution $\bar{v}$, and the sequence $\{ \bar{v}^{(S)}_h \}_h$ has a subsequence that converges to $\bar{v}$ in $H^1_\pi$.
\end{thm}

\begin{proof}
Recalling the estimate \eqref{eq:bending-0}, we have
\begin{align}
B_h[v_h] + T[v_h]
& = \sum_{j=1}^{N} \left[ \frac{C}{2} \frac{\theta_j^2}{l_j + l_{j+1}} + \sigma (l_j + l_{j+1}) \right] \\
& \ge \sum_{j=1}^{N} 2 \sqrt{\frac{C \sigma}{2} \theta_j^2} \\
& \ge \sqrt{2C\sigma} \left| \sum_{j=j_0}^{j_1-1} \theta_j \right|  \\
& =  \sqrt{2C\sigma} | \phi_{j_1} - \phi_{j_0} |
\label{eq:sufficient-3}
\end{align}
for each $v_h \in V_h$, where $j_0$ and $j_1$ are arbitrary indices and $\phi_j = \psi(v_h) = \arctan d_{h,j}(v_h)$.
If $j_1 \in J_{h,\delta} = \{ j=1,\dots,N \mid \zeta_{\delta,j-1} \zeta_{\delta,j} > 0 \}$ (cf.~Lemma \ref{lem:boundedness}), then
\begin{equation}
|d_{h,j_1}(v_h)| \le |\psi|_{W^{1,\infty}(0,1)} + 2c_0,
\end{equation}
owing to \eqref{eq:sufficient-1} and $\delta \le c_0 h$, which implies
\begin{equation}
|\phi_{j_1}(v_h)| \le \arctan ( |\psi|_{W^{1,\infty}(0,1)} + 2c_0)
\label{eq:sufficient-4}
\end{equation}
for $v_h \in V_h$.
Therefore, owing to the equations \eqref{eq:sufficient-2}, \eqref{eq:boundedness-1}, \eqref{eq:sufficient-3}, and \eqref{eq:sufficient-4},
\begin{equation}
|\phi_{j_0}(\bar{v}^{(S)}_h)| \le
\frac{1}{\sqrt{2C\sigma}} \left[ \sigma + 4\gamma \left( \frac{T[\psi]}{\sigma} + c_0 \right) \right]
+ \arctan ( |\psi|_{W^{1,\infty}(0,1)} + 2c_0).
\end{equation}
for arbitrary index $j_0$.
Hence, the assumption \eqref{eq:sufficient} implies the desired estimate \eqref{eq:sufficient-0}.
Combining \eqref{eq:sufficient-0} with Lemma \ref{lem:boundedness}, we can obtain $W^{1,\infty}$-boundedness of the discrete minimizers, which yields the global optimization \eqref{eq:problem0}.
\end{proof}

\section{Numerical examples}
\label{section:num_exp}
In this section, some numerical examples are presented.
Since the problems \eqref{eq:disc-problem} and \eqref{eq:disc-problem-S} are optimization problems in finite-dimensional spaces, we can apply various numerical algorithms to solve them.
We solve the problem \eqref{eq:disc-problem} only for simplicity.
The algorithm we choose is the quasi-Newton method with the Broyden-Fletcher-Goldfarb-Shannon (BFGS) formula \cite{NocW06}.
We stop the quasi-Newton iteration if the functional $E_{h,\delta,\rho}$ satisfies
\begin{equation}
\left\| \frac{\nabla E_{h,\delta,\rho}}{E_{h,\delta,\rho}} \right\|_\infty \le 10^{-5}
\end{equation}
as a function defined on $\bR^N$.
Note that the quasi-Newton method computes not only global minimizers but also local minimizers.
Thus, the following numerical results are merely local minimizers.
We regard the one whose energy is less than any of the others as the global minimizer.

\subsection{Sinusoidal obstacle}
\label{subsec:sinusoidal}
We choose $\psi_1(x) = 0.03 \sin(24 \pi x)$ as an obstacle.
We consider the two pairs of physical parameters as shown in Table~\ref{tab:parameters-sinusoidal}.
Moreover, we set $h=\delta=1/N$ and $\rho=h/100$ for $N=100, 200, 400$ as discretization parameters.
Then, we obtain six typical examples of local minimizers, as plotted in Figure~\ref{fig:sinusoidal}.
In addition to these local minimizers, some combinations can be local minimizers, as well.
Note that the combination must not be a global minimizer in most cases.

\begin{table}[htb]\centering
\begin{tabular}{c|ccc}
 & $C/2$ & $\sigma$ & $\gamma$ \\\hline
Parameter 1 & 0.0005 & 0.01 & 1 \\\hline
Parameter 2 & 0.0003 & 0.01 & 2 \\\hline
\end{tabular}
\caption{Physical parameters for $\psi_1$.}
\label{tab:parameters-sinusoidal}
\end{table}

\begin{figure}[htb]\centering
\begin{tabular}{ccc}
\includegraphics[page=6]{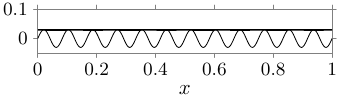}
&~&
\includegraphics[page=7]{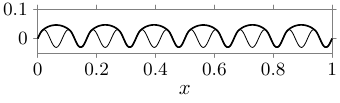} \\[-1ex]
Type A &~& Type B \\[2ex]
\includegraphics[page=8]{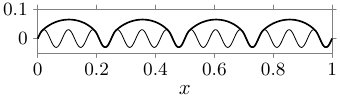}
&~&
\includegraphics[page=9]{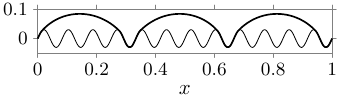} \\[-1ex]
Type C &~& Type D \\[2ex]
\includegraphics[page=10]{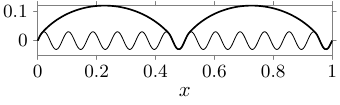}
&~&
\includegraphics[page=11]{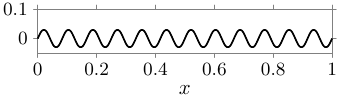} \\[-1ex]
Type E &~& Type F
\end{tabular}
\caption{Typical examples of local minimizers for the obstacle $\psi_1$. In each figure, the thin line expresses the obstacle and the thick line expresses a local minimizer. We plot the solutions for $N=400$ and Parameter~1.}
\label{fig:sinusoidal}
\end{figure}

We show the energy for each parameter and each local minimizer in Table~\ref{tab:sinusoidal-energy}.
For Parameter~2, the global minimizer is Type F, which adheres to the obstacle everywhere.
This is because the adhesion coefficient $\gamma$ is much larger than other parameters.
For Parameter~1, the global minimizer is not the trivial case (Type A nor Type F) but is a non-trivial one, Type B.
Although we could not find parameters for which Type C or Type D is the global minimizer, those examples would be quite exciting if they exist.

\begin{table}[htb]\centering\footnotesize$
\begin{array}{|c|c|c|c|c|c|c|}\hline
 & \text{Type A} & \text{Type B} & \text{Type C} & \text{Type D} & \text{Type E} & \text{Type F} \\\hline
  & -0.0900541 & \bm{-0.2465327} & -0.2007124 & -0.1404174  & -0.1036022 & 0.2670607 \\
\text{Parameter 1} & -0.0794536 & \bm{-0.2064951} & -0.1570993 & -0.1187022 & -0.0808694 & 0.4065035 \\
  & -0.0627887 & \bm{-0.1648750} & -0.1270884 &
  -0.0986503 & -0.0655316 & 0.5285373 \\\hline
  & -0.2105891 & -1.3705604 & -0.9284880 & -0.6811337 & -0.4637541 & \bm{-2.2685733} \\
\text{Parameter 2} & -0.1785572 & -1.3498736 & -0.8945974 & -0.6749780 & -0.4460241 & \bm{-2.1895721} \\
  & -0.1388582 & -1.3196272 & -0.8848577 & -0.6623300 & -0.4400052 & \bm{-2.1441625} \\\hline
\end{array}$
\caption{Energy for each parameter and local minimizer.
In each cell, the first row is the value when $N=100$, the second is for $N=200$, and the third is for $N=400$.
The bold letters describe the global minimizers.}
\label{tab:sinusoidal-energy}
\end{table}

\subsection{Almost singular obstacle}
\label{subsec:singular}
We next choose
\begin{equation}
\psi_2(x) = \frac{\varepsilon^2 x^2 (1-x)^2}{\varepsilon^2 + (2x-1)^2},
\quad \varepsilon = 0.01
\end{equation}
as an obstacle, which is smooth but has a sharp peak at $x=1/2$.
We consider three pairs of parameters as given in Table~\ref{tab:parameters-singular}.
Moreover, we set $h=\delta=1/N$ and $\rho=h/1000$ for $N=100, 200, 400$ as discretization parameters.
Then, we obtain four types of local minimizers, as plotted in Figure~\ref{fig:singular}.

\begin{table}[htb]\centering
\begin{tabular}{c|ccc}
 & $C/2$ & $\sigma$ & $\gamma$ \\\hline
Parameter 1 & 0.1 & 1 & 1 \\\hline
Parameter 2 & 0.1 & 1 & 0.01 \\\hline
Parameter 3 & 0.001 & 1 & 5 \\\hline
\end{tabular}
\caption{Physical parameters for $\psi_2$.}
\label{tab:parameters-singular}
\end{table}

\begin{figure}[htb]
\centering
\begin{tikzpicture}[x=1ex,y=1ex,every node/.style={inner sep=0ex, outer sep=0ex}]
\coordinate (O) at (0,0);
\coordinate (X) at (45,0);
\coordinate (Y) at (0,16.5);
\coordinate (C) at ($(O)!0.5!(X)$);
\coordinate (B) at ($(C) + (Y)$);
\coordinate (A) at ($(B) + (Y)$);
\coordinate (D) at ($(O)!1.5!(X)$);
\node[above] at (A) {\includegraphics[page=1]{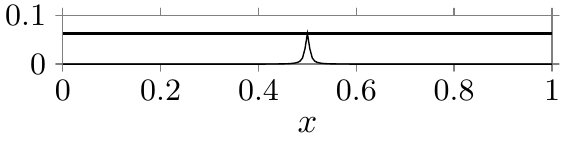}};
\node[above] at (B) {\includegraphics[page=2]{plotresults.pdf}};
\node[above] at (C) {\includegraphics[page=3]{plotresults.pdf}};
\node[above] at (D) {\includegraphics[page=4]{plotresults.pdf}};
\node[below] at (A) {Type A};
\node[below] at (B) {Type B};
\node[below] at (C) {Type C};
\node[below] at (D) {Type D};
\end{tikzpicture}
\caption{Local minimizers for $\psi_2$.
In each figure, the thin line expresses the obstacle and the thick line expresses a local minimizer. We plot the solutions for $N=400$. Types A and D are the minimizers for Parameter~2, Type B is the one for Parameter~1, and Type C is the one for Parameter~3.}
\label{fig:singular}
\end{figure}

It is quite remarkable that there exists a local minimizer, such as Type D, that appears to have singularity.
Furthermore, the $W^{1,\infty}$-seminorm of such a solution appears to be proportional to $N$ (Table~\ref{tab:singular-typeD}).
We do not know why this phenomenon occurs.
However, we imagine that there exists a curve that is a critical point of $E_{h,\delta,\rho}$ and that contains a loop, and that the ``singular'' solution corresponds to the looped curve.

\begin{rem}\label{rem:singular}
Although we could not find parameters for which Type D is a global solution, there might exist such parameters.
If they exist, then we can show that the global optimization problem \eqref{eq:cont-problem} does not have solutions in general.
\end{rem}

\begin{table}[htb]
\centering
\begin{tabular}{c|ccc}
 & $N=100$ & $N=200$ & $N=400$ \\\hline
Parameter 1 & $67.991211$ & $138.22525$ & $278.70029$ \\\hline
Parameter 2 & $68.027096$ & $138.25028$ & $278.71791$ \\\hline
\end{tabular}
\caption{Behavior of the $W^{1,\infty}$-seminorms of the local minimizers of Type D.}
\label{tab:singular-typeD}
\end{table}

We show the energy for each parameter and each local minimizer in Table~\ref{tab:singular-energy}.
We observe the local minimizer of Type B.
If $\gamma$ is large, then the solution tends to adhere to the obstacle.
Thus, the ``skirt'' of the solution becomes narrower when the value of $\gamma$ becomes larger (Figure~\ref{fig:singular-typeB}).
In addition, when $\gamma$ is small, the skirt is wide.
Thus, there exists no local minimizer of Type B if $\gamma$ is small to some degree (Table~\ref{tab:singular-energy}, Parameter~2).

\begin{table}[htb]
\centering
$\begin{array}{|c|c|c|c|c|}\hline
& \text{Type A} & \text{Type B} & \text{Type C} & \text{Type D} \\\hline
                & 1.0 & \bm{0.7915842} & & 4.6643686 \\
\text{Parameter 1} & 1.0 & \bm{0.8262222} & \times & 4.7157447 \\
                & 1.0 & \bm{0.8501729} & & 4.7457190 \\\hline
                & \bm{1.0} &        & & 4.9166234 \\
\text{Parameter 2} & \bm{1.0} & \times & \times & 4.9479670 \\
                & \bm{1.0} &        & & 4.9612867 \\\hline
                & 1.0 & -3.2773058 & \bm{-4.0788871} &        \\
\text{Parameter 3} & 1.0 & -3.2958225 & \bm{-3.9828137} & \times \\
                & 1.0 & -3.2943217 & \bm{-3.6601521} &        \\\hline
\end{array}$
\caption{Energy for each parameter and local minimizer.
In each cell, the first row is the value when $N=100$, the second for $N=200$, and the third for $N=400$.
The bold letters describe the global minimizers.
The symbol $\times$ expresses non-existence of the local minimizers.}
\label{tab:singular-energy}
\end{table}

\begin{figure}[htb]
\centering
\begin{tabular}{ccc}
\includegraphics[page=2]{plotresults.pdf}
& ~ &
\includegraphics[page=5]{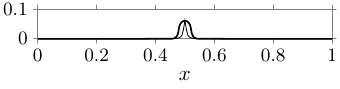} \\
Parameter 1
& &
Parameter 3
\end{tabular}
\caption{Solutions of Type B for Parameters~1 and 3.}
\label{fig:singular-typeB}
\end{figure}

\section{Concluding remarks}
\label{section:conclusion}
One of the aims of this study is to prove that there exists at least one global minimizer of the functional $E$.
However, we obtained only the existence result under the ``bounded slope condition,'' as shown in Theorem \ref{thm:compactness-S}.
The existence of global solutions was shown only in the special case (Theorem \ref{thm:sufficient}).
Moreover, as results of numerical experiments, we were able to find a ``singular'' local minimizer such as Type D in Figure~\ref{fig:singular}.
If there exists a pair of parameters such that Type D is a global minimizer, then we can show that the problem \eqref{eq:cont-problem} does not have solutions in general (Remark \ref{rem:singular}).
Therefore, it is important and challenging to find such parameters or to show that \eqref{eq:cont-problem} always has at least one solution.
We leave these problems for future works.

\section*{Acknowledgements}
I would like to thank Mr.~Tatsuya Miura for bringing this topic to my attention and encouraging me through valuable discussions.
In particular, he suggested me that there may be singular examples as presented in Subsection~\ref{subsec:singular}.
This work was supported by the Program for Leading Graduate Schools, MEXT,
Japan, and by JSPS KAKENHI (15J07471).

\bibliographystyle{plain}
\bibliography{obstacle_adhesion}
\end{document}